\documentclass[preprint,11pt]{elsarticle}

\usepackage{amssymb,amscd,amsmath,amsthm,color}

\usepackage[T1]{fontenc}
\usepackage{lmodern}
\usepackage{amsmath,amssymb,amsthm,mathtools}
\usepackage[expansion=false]{microtype}
\usepackage{xurl}
\usepackage{needspace}
\usepackage[hidelinks]{hyperref}
\newtheorem{theorem}{Theorem}[section]
\newtheorem{corollary}[theorem]{Corollary}
\newtheorem{proposition}[theorem]{Proposition}
\numberwithin{equation}{section}
\newcommand{\Md}{\mathbb M_d}
\newcommand{\Mdh}{\mathbb M_d^{\mathrm h}}
\newcommand{\Mdp}{\mathbb M_d^+}
\newcommand{\Tr}{\operatorname{Tr}}
\newcommand{\ran}{\operatorname{Ran}}
\newcommand{\rank}{\operatorname{rank}}
\newcommand{\diag}{\operatorname{diag}}
\newcommand{\one}{\mathbf 1}

\newcommand{\join}{\vee}
\newcommand{\meet}{\wedge}
\newcommand{\bigjoin}{\mathop{\bigvee}}
\makeatletter
\let\ps@pprintTitle\ps@empty
\makeatother
\begin{document}
\begin{frontmatter}
\title{A Grassmann formula for the spectral order of matrices}
\author{Mohamed Amine Aouichaoui}
\ead{amine.aouichaoui@fsm.rnu.tn}
\address{Department of Mathematics, Faculty of Sciences of Monastir,\newline
University of Monastir, Monastir, Tunisia}

\author{Eun-Young Lee}
\ead{eylee89@knu.ac.kr}
\address{Department of Mathematics, KNU-Center for Nonlinear Dynamics,
Kyungpook National University,
 Daegu 702-701, Korea}

\begin{abstract}
For Hermitian matrices $A$ and $B$, we prove that
$(A\join B)\oplus(A\meet B)$ is unitarily equivalent to $A\oplus B$,
where the join and meet are taken in Olson's spectral order.
Taking traces answers  a question of Bourin and Lee: for positive
semidefinite matrices, $\Tr(A\join B)=\Tr(A+B)$ if and only if
$A\meet B=0$. Iterating the direct-sum identity gives a trace formula
for a finite family of positive matrices. The trace of the spectral
supremum equals the trace of the sum precisely when the ranges form
an algebraic direct sum. For each fixed $1<p<\infty$, the spectral supremum and the sum have
equal Schatten $p$-norms if and only if the ranges are pairwise
orthogonal.  Finally we establish an elegant  formula
for the Frobenius inner product and the spectral order.
\end{abstract}

\begin{keyword}
Spectral order \sep Positive matrix \sep Unitary equivalence \sep
Trace equality \sep Schatten norm
\MSC[2020] 15A18 \sep 15A42 \sep 15A60
\end{keyword}
\end{frontmatter}

\section{Introduction}\label{sec:introduction}

Bourin and Lee~\cite[Corollary~3.8]{BL} proved  the weak majorization
\begin{equation}\label{eq:known-comparison}
 A\join B\prec_w A+B
\end{equation}
for positive semidefinite matrices $A$ and $B$, where $A\join B$
is their supremum in Olson's spectral order~\cite{Olson}. This  natural order endows the set of Hermitian operators with a lattice structure,
extending the lattice of projections,
and \eqref{eq:known-comparison} means that $$\|A\join B\|\leq\|A+B\|$$ for every unitarily invariant (or symmetric) norm, thereby completing a series of majorizations due to Ando \cite[Section 6]{Ando}.

Bourin and Lee's  question at the end of their Example 3.4 (Question 3.5 in the Arxiv version) asks whether
\begin{equation}\label{eq:trace-question}
 \Tr(A\join B)=\Tr(A+B)
 \quad\Longleftrightarrow\quad A\meet B=0.
\end{equation}
Here $A\meet B$ is the infimum in the same order. As noted in that
question, $A\meet B=0$ is equivalent to
$\ran A\cap\ran B=\{0\}$.

We prove that the join and meet together have the same eigenvalues,
with the same multiplicities, as the original pair. More precisely,
for any two Hermitian matrices,
\begin{equation}\label{eq:intro-identity}
 (A\join B)\oplus(A\meet B)\simeq A\oplus B.
\end{equation}
 where $\simeq$ stands for the unitarily congruence relation. Taking two projections and traces, we have the classical Grassmann dimension formula, and our
 proof uses this formula  applied to
spectral subspaces. The resulting trace identity
\[
 \Tr(A+B)-\Tr(A\join B)=\Tr(A\meet B)
\]
proves~\eqref{eq:trace-question} for positive matrices.

For a finite family, successive joins and meets give a corresponding
direct-sum identity. For positive matrices, it shows that the trace
of the supremum equals the trace of the sum exactly when the ranges
form an algebraic direct sum. We also treat equality of Schatten
$p$-norms. When $1<p<\infty$, equal norms instead force the ranges
to be pairwise orthogonal. The proof combines the direct-sum identity
with the classical pinching formula~\cite[Section~1]{Pinching}
and the strict convexity of Schatten norms~\cite{BCL}.

We write $\Md$ for the complex $d\times d$ matrices, $\Mdh$ for
the Hermitian matrices, and $\Mdp$ for the positive semidefinite
matrices. All traces are unnormalized. The notation $X\simeq Y$
means that $X=UYU^*$ for a unitary matrix $U$. For Hermitian
matrices $X,Y\in\mathbb M_n$, the relation $X\prec_w Y$ means
\[
 \sum_{j=1}^k\lambda_j(X)\leq\sum_{j=1}^k\lambda_j(Y)
 \quad(1\leq k\leq n),
\]
where eigenvalues are listed in nonincreasing order. For
$1\leq p<\infty$, the Schatten $p$-norm is
\[
 \|X\|_p=(\Tr|X|^p)^{1/p},
 \qquad |X|=(X^*X)^{1/2}.
\]
We use the standard results on majorization and unitarily invariant
norms in~\cite[Chapters~II and~IV]{Bhatia}.

Section~\ref{sec:trace} proves the direct-sum identity \eqref{eq:intro-identity}, extends it to
finite families, and characterizes trace equality for positive matrices.
Section~\ref{sec:norms} treats equality in Schatten norms and gives an
exact formula for the difference of the squared Hilbert--Schmidt norms.

\section{Spectra and trace equalities}\label{sec:trace}

For $X\in\Mdh$ and $t\in\mathbb R$, put
$E_X(t)=\one_{(t,\infty)}(X)$,  the spectral projection of $X$ corresponding to eigenvalues $>t$. The spectral order is defined by
\[
 A\preceq B
 \quad\Longleftrightarrow\quad
 E_A(t)\leq E_B(t)\quad(t\in\mathbb R).
\]
For orthogonal projections $P,Q$, the projections $P\join Q$ and
$P\meet Q$ have ranges $\ran P+\ran Q$ and
$\ran P\cap\ran Q$, respectively.
In finite dimension, Olson's lattice formulas give
\begin{equation}\label{eq:spectral-formulas}
 \begin{aligned}
 E_{A\join B}(t)&=E_A(t)\join E_B(t),\\
 E_{A\meet B}(t)&=E_A(t)\meet E_B(t),
 \qquad t\in\mathbb R.
 \end{aligned}
\end{equation}
See~\cite{Olson} and~\cite[Remark~3.3]{BL}. The positive-matrix
formulas extend to Hermitian matrices by adding a sufficiently large
scalar multiple of the identity. Since matrix spectral families are
step functions, both formulas hold at the same threshold $t$.
For positive matrices they also give the familiar range formulas
\begin{equation}\label{eq:range-formulas}
 \ran(A\join B)=\ran A+\ran B,
 \qquad
 \ran(A\meet B)=\ran A\cap\ran B.
\end{equation}

\begin{theorem}\label{thm:two}
For every $A,B\in\Mdh$,
\begin{equation}\label{eq:two-identity}
 (A\join B)\oplus(A\meet B)\simeq A\oplus B.
\end{equation}
\end{theorem}

\begin{proof}
Fix $t\in\mathbb R$ and let $P=E_A(t)$ and $Q=E_B(t)$.
The Grassmann dimension formula gives
\[
 \rank(P\join Q)+\rank(P\meet Q)=\rank P+\rank Q.
\]
By~\eqref{eq:spectral-formulas},
\[
 \rank E_{A\join B}(t)+\rank E_{A\meet B}(t)
 =\rank E_A(t)+\rank E_B(t).
\]
For a Hermitian matrix $X$, the number $\rank E_X(t)$ counts its
eigenvalues strictly larger than $t$. The two matrices
in~\eqref{eq:two-identity} therefore have the same eigenvalue-counting
function on $\mathbb R$. Their eigenvalues agree, including
multiplicities, and the spectral theorem gives the unitary equivalence.
\end{proof}

In particular, the spectra of $A\join B$ and $A\meet B$ lie in
$\sigma(A)\cup\sigma(B)$. Applying any complex-valued function $f$
on this finite set to~\eqref{eq:two-identity} gives
\[
 f(A\join B)\oplus f(A\meet B)\simeq f(A)\oplus f(B).
\]
Thus
\begin{equation}\label{eq:functional-trace}
 \Tr f(A\join B)+\Tr f(A\meet B)=\Tr f(A)+\Tr f(B).
\end{equation}

\begin{corollary}\label{cor:two-trace}
For $A,B\in\Mdp$,
\begin{equation}\label{eq:two-trace}
 \Tr(A+B)-\Tr(A\join B)=\Tr(A\meet B).
\end{equation}
Consequently,
\[
 \Tr(A\join B)=\Tr(A+B)
 \quad\Longleftrightarrow\quad A\meet B=0
 \quad\Longleftrightarrow\quad \ran A\cap\ran B=\{0\}.
\]
\end{corollary}

\begin{proof}
Take $f(t)=t$ in~\eqref{eq:functional-trace}. Since $A\meet B$ is
positive, its trace is zero exactly when it is the zero matrix.
The last equivalence is the one recalled
in~\cite[Question~3.5]{BL}.
\end{proof}

We next apply Theorem~\ref{thm:two} to a finite family. Let $m\geq2$
and let $A_1,\ldots,A_m\in\Mdh$. Define
\begin{equation}\label{eq:recursion}
 C_1=A_1,
 \qquad C_j=C_{j-1}\join A_j,
 \qquad D_j=C_{j-1}\meet A_j
 \quad(2\leq j\leq m).
\end{equation}
Then $C_m=\bigjoin_{j=1}^m A_j$.

\begin{corollary}\label{cor:family}
With the notation in~\eqref{eq:recursion},
\begin{equation}\label{eq:family-identity}
 C_m\oplus D_2\oplus\cdots\oplus D_m
 \simeq A_1\oplus\cdots\oplus A_m.
\end{equation}
For every complex-valued function $f$ on
$\bigcup_{j=1}^m\sigma(A_j)$,
\begin{equation}\label{eq:family-functional}
 \Tr f(C_m)+\sum_{j=2}^m\Tr f(D_j)
 =\sum_{j=1}^m\Tr f(A_j).
\end{equation}
\end{corollary}

\begin{proof}
At each step, Theorem~\ref{thm:two} gives
\[
 C_j\oplus D_j\simeq C_{j-1}\oplus A_j.
\]
Starting with $j=m$, substitute these identities successively and
permute the direct-sum blocks. This gives~\eqref{eq:family-identity}.
Its left-hand side has no eigenvalues outside
$\bigcup_j\sigma(A_j)$. Applying $f$ and taking traces
proves~\eqref{eq:family-functional}.
\end{proof}

\Needspace{7\baselineskip}
\begin{corollary}\label{cor:family-trace}
Let $A_1,\ldots,A_m\in\Mdp$, and write
$C=C_m$ and $S=\sum_{j=1}^m A_j$. Then
\begin{equation}\label{eq:trace-difference}
 \Tr S-\Tr C=\sum_{j=2}^m\Tr D_j.
\end{equation}
Moreover, $\Tr C=\Tr S$ if and only if $D_2=\cdots=D_m=0$,
or equivalently, if and only if $\ran A_1+\cdots+\ran A_m$
is an algebraic direct sum.
\end{corollary}

\begin{proof}
Equation~\eqref{eq:trace-difference} follows by taking $f(t)=t$
in~\eqref{eq:family-functional}. All the $D_j$ are positive, so the
right-hand side vanishes exactly when each $D_j$ vanishes.
The range formulas~\eqref{eq:range-formulas} yield
\[
 \ran C_j=\sum_{i=1}^j\ran A_i,
 \qquad
 \ran D_j=
 \left(\sum_{i=1}^{j-1}\ran A_i\right)\cap\ran A_j.
\]
Thus all the $D_j$ vanish precisely when the sum of the ranges is
a direct sum.
\end{proof}

The direct sum in this corollary need not be orthogonal. For three
or more matrices, pairwise zero intersections do not suffice. For
example, take the rank-one projections
\[
 P=\begin{pmatrix}1&0\\0&0\end{pmatrix},
 \qquad Q=\begin{pmatrix}0&0\\0&1\end{pmatrix},
 \qquad R=\frac12\begin{pmatrix}1&1\\1&1\end{pmatrix}.
\]
Their ranges are three distinct lines in $\mathbb C^2$, so every
pair of ranges has zero intersection. Nevertheless,
$P\join Q\join R=I_2$ and
\[
 \Tr(P\join Q\join R)=2<3=\Tr(P+Q+R).
\]

\section{Schatten $p$-norms and spectral order}\label{sec:norms}

For a matrix $H$ partitioned into $m$ square diagonal blocks of the
same size, let $\mathcal E_m(H)$ denote its block-diagonal part.
The classical pinching formula expresses $\mathcal E_m(H)$ as an
average of $m$ unitary conjugates of $H$, one of which is $H$
itself~\cite[Section~1]{Pinching}. In particular,
\[
 \|\mathcal E_m(H)\|_p\leq\|H\|_p
 \qquad(1\leq p<\infty).
\]
For $1<p<\infty$, the Schatten $p$-norm is strictly
convex~\cite{BCL}. Equality in this pinching inequality therefore
holds exactly when $H=\mathcal E_m(H)$.

\Needspace{7\baselineskip}
\begin{theorem}\label{thm:schatten}
Let $m\geq2$, $A_1,\ldots,A_m\in\Mdp$,
$C=\bigjoin_{j=1}^m A_j$, and $S=\sum_{j=1}^m A_j$.
For each fixed $1<p<\infty$,
\[
 \|C\|_p=\|S\|_p
 \quad\Longleftrightarrow\quad
 A_iA_j=0\quad(i\ne j).
\]
These conditions are also equivalent to $C=S$.
\end{theorem}

\begin{proof}
Use the matrices $D_j$ from~\eqref{eq:recursion}, and set
\[
 T=\begin{pmatrix}A_1^{1/2}&\cdots&A_m^{1/2}\end{pmatrix},
 \qquad H=T^*T,
 \qquad L=A_1\oplus\cdots\oplus A_m.
\]
Here $T$ has size $d\times md$. Since $TT^*=S$, the singular-value
decomposition gives
\[
 H\simeq S\oplus0_{(m-1)d},
 \qquad \mathcal E_m(H)=L.
\]
By~\eqref{eq:family-identity},
$\|L\|_p^p=\|C\|_p^p+\sum_{j=2}^m\|D_j\|_p^p$.
It follows that
\begin{equation}\label{eq:norm-chain}
 \|C\|_p\leq\|L\|_p
 =\|\mathcal E_m(H)\|_p
 \leq\|H\|_p=\|S\|_p.
\end{equation}
If $\|C\|_p=\|S\|_p$, equality holds throughout~\eqref{eq:norm-chain}.
The equality condition for pinching gives $H=\mathcal E_m(H)$.
Its off-diagonal blocks must therefore vanish:
\[
 A_i^{1/2}A_j^{1/2}=0\qquad(i\ne j).
\]
Multiplication on the left by $A_i^{1/2}$ and on the right by
$A_j^{1/2}$ gives $A_iA_j=0$ for $i\ne j$.

Conversely, if $A_iA_j=0$ for $i\ne j$, the ranges of the $A_j$
are pairwise orthogonal. Hence,
for every $t\geq0$,
\[
 E_S(t)=\sum_{j=1}^m E_{A_j}(t)
       =\bigjoin_{j=1}^m E_{A_j}(t)=E_C(t).
\]
Thus $C=S$, which in turn implies equality of the norms.
\end{proof}

For $p=1$, equality is characterized by
Corollary~\ref{cor:family-trace}, since all matrices involved are
positive. For $1<p<\infty$, Theorem~\ref{thm:schatten} shows that
equality for one exponent gives equality for every unitarily
invariant norm: the two matrices are then equal.
The upper restriction on $p$ matters. If
$A=\diag(2,1)$ and $B=\diag(0,1)$, then $A\join B=A$ and
\[
 \|A\join B\|_\infty=\|A+B\|_\infty=2,
 \qquad AB\ne0,
\]
where $\|\cdot\|_\infty$ denotes the operator norm.

For the Hilbert--Schmidt norm, the difference can be computed from
the successive meets. Taking $f(t)=t^2$
in~\eqref{eq:family-functional} and expanding $\Tr S^2$ gives
\begin{equation}\label{eq:hilbert-schmidt}
 \|S\|_2^2-\|C\|_2^2
 =\sum_{j=2}^m\|D_j\|_2^2
   +2\sum_{1\leq i<j\leq m}
          \|A_i^{1/2}A_j^{1/2}\|_2^2.
\end{equation}
Indeed, the terms involving two distinct indices in the expansion
are $2\Tr(A_iA_j)=2\|A_i^{1/2}A_j^{1/2}\|_2^2$.
For two positive matrices, this reads as an elegant formula for their Frobenius inner product in terms of the spectral order:

\begin{proposition} Let $A,B\in\Mdp$. then, 
\[
\Tr(AB)=\frac{1}{2}\left\{
 \|A+B\|_2^2-\|A\join B\|_2^2
 -\|A\meet B\|_2^2\right\}.
\]
\end{proposition}

To see the distinction between trace equality and Schatten norm
equality in dimension two, let $0<c<1$ and take
\[
 P=\begin{pmatrix}1&0\\0&0\end{pmatrix},
 \qquad
 Q=\begin{pmatrix}
 c^2&c\sqrt{1-c^2}\\
 c\sqrt{1-c^2}&1-c^2
 \end{pmatrix}.
\]
These are rank-one projections onto distinct, nonorthogonal lines.
Thus $P\join Q=I_2$ and $P\meet Q=0_2$, whereas $P+Q$ has
eigenvalues $1+c$ and $1-c$. The traces are equal, but for every
$p>1$ strict convexity of $t^p$ gives
\[
 \|P\join Q\|_p^p=2<(1+c)^p+(1-c)^p=\|P+Q\|_p^p.
\]
At $p=2$, the difference is $2c^2=2\Tr(PQ)$,
as in~\eqref{eq:hilbert-schmidt}.


\begin{thebibliography}{9}

\bibitem{Ando} 
T.\ Ando, Majorization, doubly stochastic matrices, and comparison of eigenvalues, Lin. Alg.
Appl. 118 (1989), 163--248. {\nolinkurl{doi:10.1016/0024-3795(89)90580-6}}

\bibitem{BCL}
K. Ball, E. A. Carlen, E. H. Lieb,
Sharp uniform convexity and smoothness inequalities for trace norms,
Invent. Math. 115 (1994), 463--482.
\href{https://doi.org/10.1007/BF01231769}{\nolinkurl{doi:10.1007/BF01231769}}.

\bibitem{Bhatia}
R. Bhatia,
\emph{Matrix Analysis},
Graduate Texts in Mathematics, vol.~169,
Springer, New York, 1997.
\href{https://doi.org/10.1007/978-1-4612-0653-8}{\nolinkurl{doi:10.1007/978-1-4612-0653-8}}.

\bibitem{Pinching}
R. Bhatia,
Pinching, trimming, truncating, and averaging of matrices,
Amer. Math. Monthly 107 (2000), no.~7, 602--608.
\href{https://doi.org/10.1080/00029890.2000.12005245}{\nolinkurl{doi:10.1080/00029890.2000.12005245}}.

\bibitem{BL}
J.-C. Bourin, E.-Y. Lee,
\emph{Averages over matrix unitary orbits and spectral order},
preprint, 2026.
\href{https://arxiv.org/abs/2606.15624v2}{arXiv:2606.15624v2}.

\bibitem{Olson}
M. P. Olson,
The selfadjoint operators of a von Neumann algebra form a
conditionally complete lattice,
Proc. Amer. Math. Soc. 28 (1971), no.~2, 537--544.
\href{https://doi.org/10.1090/S0002-9939-1971-0276788-1}{\nolinkurl{doi:10.1090/S0002-9939-1971-0276788-1}}.

\end{thebibliography}
\end{document}